\documentclass[12pt]{amsart}

\usepackage[T2A]{fontenc}     
\usepackage[utf8]{inputenc}   
\usepackage[english]{babel}   

\newtheorem{thm}{Theorem}[section]

\newtheorem{lem}[thm]{Lemma}

\theoremstyle{definition}
\newtheorem{defn}[thm]{Definition}
\newtheorem{rem}[thm]{Remark}

\numberwithin{equation}{section}

\begin{document}

\begin{center}
\textbf{{\large {\ 
Source Identification Problem for a Nonlinear Subdiffusion Equation}}}\\[0pt]
\medskip \textbf{R.R. Ashurov$^{1,2}$ and O.T. Mukhiddinova$^{3,1}$}\\[0pt]
\textit{ashurovr@gmail.com, oqila1992@mail.ru \\[0pt]}

\smallskip

\textit{$^{1}$ V.I. Romanovskiy Institute of Mathematics, Uzbekistan Academy of Science, University str., 9, Olmazor district, Tashkent, 100174, Uzbekistan} \\

\textit{$^{2}$ Central Asian University, 264 Milliy Bog Street, Barkamol MFY, Mirzo Ulugbek District, Tashkent 111221, Uzbekistan} \\

\textit{$^{3}$ Tashkent University of Information Technologies named after Muhammad al-Khwarizmi, Str., 108, Amir Temur Avenue, Tashkent, 100200, Uzbekistan}

\end{center}

\textbf{Abstract.}
The work is devoted to the study of the inverse problem of determining the right-hand side of a nonlinear subdiffusion equation with a Caputo derivative with respect to time. Nonlinearity of the equation means that the right-hand side of the equation depends nonlinearly on the solution of the equation. The inverse problem consists of reconstructing the coefficient of the right-hand side, which depends on both time and spatial variables, under a measurement in an integral form. Similar inverse problems were previously studied in the case when the right-hand side depends only on time or on a spatial variable. A weak solution is sought by the Galerkin method. A priori estimates are proved, and with their help, the existence and uniqueness of a solution to the inverse problem under consideration are established. It is noteworthy that the results obtained are new for diffusion equations as well.


\section{Introduction}

Let $Q_T = (0, T) \times G \times (0, \pi)$ and $D = G \times (0, \pi)$, where $G \subset \mathbb{R}^n$ is a bounded domain with a sufficiently smooth boundary. We consider the following initial-boundary value problem:
\begin{equation} \label{1}
\begin{cases}
D_t^{\alpha} u - \Delta_x u - u_{yy} = g(t, x, y, u) + f(t, x, y) \cdot h(t, x), & (t, x, y) \in Q_T, \\
u(0, x, y) = \varphi(x, y), & (x, y) \in D, \\
u(t, x, y) = 0, & x \in \partial G, t \in [0, T], y \in [0, \pi], \\
u(t, x, 0) = u(t, x, \pi) = 0, & t \in [0, T], x \in G.
\end{cases}
\end{equation}

Here, $f$, $g$, and $\varphi$ are given continuous functions, and $\Delta_x$ denotes the Laplacian with respect to the spatial variables $x$, $\alpha \in (0, 1]$. The operator $D_t^\alpha$ denotes the Caputo fractional derivative, which for an absolutely continuous function $v(t)$ is defined by (see, e.g., \cite{Kil}, p. 91)
$$
D_t^\alpha v(t) := J_t^{1 - \alpha} D_t v(t), \quad J_t^\alpha v(t) = \frac{1}{\Gamma(\alpha)} \int_0^t (t - \tau)^{\alpha - 1} v(\tau)\, d\tau,
$$
where $D_t = d/dt$, $\Gamma(\alpha)$ is the Gamma function, and $J_t^\alpha$ is the Riemann–Liouville fractional integral operator. In what follows we will denote the operators $\Delta_x$ and $\nabla_x$ as $\Delta$ and $\nabla$; it is clear that this will not cause confusion.

If the function $h(t, x)$ is known, then under appropriate assumptions on the given data, the initial-boundary value problem \eqref{1} admits a unique solution (see, e.g., \cite{Gal}).

Now, suppose that $h(t, x)$ is unknown and must be recovered. The main objective of this paper is to investigate the inverse problem of simultaneously determining the pair $\{ u(t, x, y), h(t, x) \}$, subject to the additional integral condition:
\begin{equation} \label{overdetermination}
\int_0^\pi u(t, x, y) \omega(y) dy = H(t, x), \quad t \in [0, T],\ x \in G,
\end{equation}
where $\omega$ and $H$ are known functions; their precise assumptions will be specified later.

Fractional differential equations provide a more adequate framework for modeling processes that exhibit memory effects, as they naturally incorporate nonlocal dependencies through fractional-order derivatives. Unlike classical differential equations, which describe dynamics based solely on instantaneous states, fractional models account for the cumulative influence of past states, making them particularly suitable for systems in physics, biology, and engineering where historical behavior significantly impacts future outcomes (see, e.g. \cite{Machado}). 

The two most commonly employed fractional derivatives are the Riemann-Liouville and Caputo derivatives (see, e.g. \cite{Kil}). The Caputo derivative, first introduced in \cite{Caputo} to analyze energy dissipation in anelastic materials, has become a valuable tool in engineering applications. Its use is supported, for instance, by the generalized Langevin equation \cite{13} and specific limiting processes tied to probability theory \cite{14, 15}. Unlike Riemann-Liouville derivatives, Caputo derivatives mitigate singularities at the origin and share properties with classical derivatives, making them particularly well-suited for initial value problems (see, e.g. \cite{Samko, PSK}).

In recent years, there has been a growing interest in inverse problems for both classical and fractional-order differential equations (see, e.g., the monograph by S. I. Kabanikhin \cite{Kaban} and the review by Yamamoto \cite{Yamamoto1}). This interest is motivated by the central role such problems play in applications ranging from mechanics and seismology to medical imaging and geophysics (see, e.g., \cite{Machado}, \cite{Kaban}, \cite{Yamamoto1}).

A considerable amount of research has focused on inverse problems for source identification when the source term has the separable form $F(x, t) = g(t) f(x)$, where either $f(x)$ or $g(t)$ is unknown. To the best of our knowledge, the more general case in which such a factorization is not assumed remains largely unexplored. In this situation, even the formulation of a suitable overdetermination condition becomes nontrivial.

Inverse problems for recovering a time-dependent factor $g(t)$ are typically approached by reducing the problem to a Volterra-type integral equation (see, e.g., \cite{ASh1}–\cite{ASh2} and references therein). On the other hand, the recovery of a spatially dependent factor $f(x)$ is commonly analyzed in two main settings: when $g(t) \equiv 1$, and when $g(t) \not\equiv 1$. The former has been studied in several works (e.g., \cite{KirM}–\cite{AMux1}), while the latter poses additional challenges, with solvability strongly influenced by the properties of $g(t)$.

To address such inverse problems, researchers often impose an overdetermination condition, either in the form of a final time measurement $u(x, T) = \psi(x)$, or as an integral condition derived from the solution $u$ (see, e.g., \cite{ASh3}–\cite{VanBockstal}, as well as \cite{Kaban} and the review \cite{Yamamoto1}).

The works most closely related to our study are those by S. Z. Dzhamalov and co-authors \cite{Djamalov1}–\cite{Djamalov3}, which address inverse problems involving the identification of the right-hand side of differential equations, where this term depends on both time and a subset of the spatial variables. These studies focus on establishing the existence of generalized solutions using the Galerkin method. 

A related inverse problem for a hyperbolic equation is studied in \cite{Fikret}, where the right-hand side also depends on time and a part of the spatial domain. The authors prove the uniqueness of the classical solution to the inverse problem.

Another notable contribution is found in \cite{XiaomaoDeng}, where a numerical method is proposed for solving a similar inverse problem for a parabolic equation.

Let us also mention two recent papers \cite{AshurovMuhiddinovanew} and \cite{AshurovMuhiddinovanew1}, where similar inverse problems are studied. In these papers, the authors studied the inverse problem of determining the right-hand side of the linear subdiffusion equation with the Caputo derivative with respect to time, where the right-hand side depends, as in the present paper, on both time and some spatial variables. The overdetermination condition in \cite{AshurovMuhiddinovanew} has the form
\begin{equation}\label{l_0}
u(t, x, l_0) = \psi(t, x), \quad t \in (0, T), \quad x \in (0, 1), \quad l_0 \in (0, \pi),
\end{equation}
where \( \psi \) is a known continuous function.
The existence and uniqueness of a weak solution of the inverse problem under consideration is proved. The authors of the work \cite{AshurovMuhiddinovanew1} investigated the inverse problem with the overdetermination condition (\ref{overdetermination}). In both works, the authors applied the Fourier method.

This paper consists of six sections. The following section contains supporting material, including well-known results of A. Alikhanov and analogy of the Aubin–Lions compactness lemma in the case of fractional
derivatives. Section 3 defines a weak solution to the inverse problem under consideration and presents the main result of the study. The solution to the problem is sought using the Galerkin method. In Section 4, a priori estimates of the solution of the inverse problem, its derivatives, and the unknown right-hand side are established. In Section 5, the convergence of the corresponding sequences is proved and, as a consequence, the main result is proved. The conclusion is given in Section 6.

\section{Preliminaries}
In this section, we remind the definition of the Mittag - Leffler functions and introduce some auxiliary lemmas that will be used throughout the paper. In particular, Alikhanov's estimates and an analogue of the Aubin–Lions lemma on compactness in the case of fractional derivatives are presented.

 The two-parameter Mittag-Leffler function $E_{\rho,\mu}(z)$ is an entire function defined by a power series of the form:
$$
E_{\rho,\mu}(z)= \sum\limits_{k=0}^\infty \frac{z^k}{\Gamma(\rho
k+\mu)},\quad \rho>0, \quad \mu, z\in \mathbb{C}.
$$
If \(\mu = 1\), then the Mittag-Leffler function is called the one-parameter or classical Mittag-Leffler function and is denoted by \(E_{\rho}(z) = E_{\rho,1}(z)\). Obviously, there is a constant $M (:=M(b))$ such that
\begin{equation}\label{E}
   E_{\rho,\mu}(z) \leq M,  \quad \mu>0, \quad z\in [0,b]\subset \mathbb{R}_+. 
\end{equation}

We also need the following connection between the fractional integral and derivative, which follows directly from the definitions and the equality for $v\in L_1(0,T)$: $J^\alpha_t(J^\beta_t v(t))= J^\beta_t(J^\alpha_t v(t))=J^{\alpha+\beta}_t v(t))$, $\alpha, \beta >0$ (see \cite{Dzh66}, p. 567).
\begin{lem}\label{J} Let $0<\alpha<1$ and $v(t)$ be absolutely continuous on $[0,T]$: $v\in AC\,[0,T]$. Then
\begin{equation}\label{JD}
    J_t^\alpha D^\alpha_t v(t)= v(t)- v(0).
\end{equation}
    
\end{lem}
Proof see in \cite{Dzh66}, p. 570.

Let us now recall the following two statements from the well-known work of Alikhanov \cite{Alixan}.
\begin{lem}\label{Alixan1}
Let $0<\alpha<1$ and  let  $w(t,x, y)$ be an absolutely continuous function on $[0,T]$ with values on $L_2(D)$. Then
\[
\frac{1}{2} D_t^\alpha \| w(t,x, y) \|_{L_2(D)}^2 \leq \int_D w(t,x, y) \, D_t^\alpha w(t,x, y) \, dx dy.
\]
\end{lem}

\begin{lem}\label{Alixan2L}
Let \( v(t)\in AC\, [0,T] \) be a positive function and for all \( t \in (0,T] \), the following inequality holds:
\[
D_t^\alpha v(t) \leq c_1 v(t) + c_2(t), \quad 0<\alpha\leq 1,
\]
for almost all $t\in [0,T]$, where \( c_1>0 \) and \( c_2(t) \) is an integrable nonnegative function on $[0,T]$. Then
\begin{equation}\label{Alixan2}
v(t)\leq v(0) \,E_\alpha(c_1 t^\alpha) + \Gamma(\alpha) \,E_{\alpha, \alpha}(c_1 t^\alpha) \, J_t^\alpha c_2(t).
    \end{equation}
\end{lem}

Let \( B \) be a Banach space. We denote by \( L_\infty(0, T; B) \) the space of functions that are essentially bounded on \( (0, T) \) and take values in \( B \). The space \( L_1(0, T; B) \) is defined similarly. Let \( W_2^k(\Omega) \) denote a classical Sobolev space, where $\Omega\subset \mathbb{R}^N$, $N\geq 1$, is a bounded domain with the smooth boundary. Then, the symbol \( \dot{W}_2^1(\Omega) \) represents the closure of the set \( C_0^\infty(\Omega) \) with respect to the norm of \( W_2^1(\Omega) \). Here and below \( \|\cdot\| \) denotes the norm, and $(\cdot, \cdot)$ denotes the scalar product in \( L_2(0, \pi) \).

In this paper, we also use the following lemma (see \cite{LiLiu}), which is analogous to the Aubin–Lions compactness lemma in the case of fractional derivatives. This lemma is applied in situations where we have a sequence of functions \( \{u^n\} \) that is bounded in some functional space with high regularity (e.g., a Sobolev space), and we aim to show that this sequence has a subsequence that converges strongly in a less restrictive space (e.g., \( L^2 \)). Let us formulate the lemma in a form convenient for us.
\begin{lem}\label{li} (\,\cite{LiLiu}, Theorem 4.2). Let $T > 0$, $\alpha \in (0, 1)$. Let $M,  B, Y$  be Banach
spaces, such that  $M \hookrightarrow \hookrightarrow  B$ compactly and $B \hookrightarrow Y$ continuously. Suppose $W \subset L_1^{loc} ((0; T ); M)$ satisfies the following:
\begin{enumerate}
    \item There exists $C_1>0$ such that for all \( u \in W \) \, one has \, \(\sup\limits_{t\in (0, T)}  J_t^\alpha(||u||^2_M) \leq C_1 \);
    \item \( W \) is bounded in \( L_2((0,T); B)\);
    \item There exists $C_2 > 0$ such that for any $u \in W$ there is an assignment of initial value $u_0$ for $u$ so that the  Caputo derivative satisfies:
    \[
    ||D^\alpha_t u||_{L_2((0,T); Y)} \leq C_2.
    \]
\end{enumerate}
Then $W$ is relatively compact in $L_2((0, T); B)$.
\end{lem}

Let us also note the following simple property of the fractional integral: Let $v(t)$ be a positive bounded function on $[0, T]$ and $\alpha\in (0,1]$. Then
\begin{equation}\label{integral}
\frac{T^{\alpha -1 }}{\Gamma(\alpha)\, }\int_0^tv(s) ds \leq J_t^\alpha v(t)\leq \frac{T^\alpha}{\alpha\, \Gamma(\alpha)} \max\limits_{t\in [0,T]} v(t).
\end{equation}

\section{Definition of the Weak Solution and Formulation of the Main Result}

Let us formulate the conditions on the functions $\omega$, $f$, $g$, and $H(t, x)$. Everywhere below, we will assume that
\begin{equation}\label{omega}
    \omega \in \dot{W}^1_2(0, \pi) \,\,\text{and} \,\int_0^\pi \varphi(x,y) \omega(y) dy= H(0, x).
\end{equation}

As for the function $f$, we will assume that it is continuous with respect to $(t, x, y) \in \overline{Q}_T$ and \( (f(t, x, \cdot), \omega) \neq 0 \) for all \( (t, x) \in [0, T] \times \overline{G} \). In what follows, we will assume that the following estimate is fulfilled: 
\begin{equation}\label{fomega}
f_\omega (t, x, y): =  \frac{f(t, x, y)}{(f(t, x, \cdot), \omega)}, \,\, \max_{ \overline{Q}_T}|f_\omega (t, x, y)|\leq \frac{1}{\sqrt{\pi}}\, f_M, 
\end{equation}
where $f_M$ is some positive constant.  

Now we formulate the conditions for the nonlinear function $g (t, x, y, u)$. First, it must be continuous in arguments $(t,x,y)\in \overline{Q}_T$ (for any $u$) and satisfy the standard Lipschitz condition:  for any $u$ and $v$ we assume, that
\begin{equation}\label{g_Lipschitz}
    |g(t,x, y, u)- g(t,x, y,  v)|\leq \ell_0  |u-v|,\,\, (t,x,y)\in \overline{Q}_T,
    \end{equation}
where $\ell_0$ is some positive constant. If we denote $_0 g (t, x, y, u)= g(t,x, y, u)- g(t, x, y, 0)$, then $_0 g(t,x, y, 0) \equiv 0,$ and 
\begin{equation}\label{g0}
g (t, x, y, u) =\,\,\, _0 g(t,x, y, u)+ g(t, x, y, 0).   \end{equation}
Lipschitz’s condition (\ref{g_Lipschitz}) implies
\begin{equation}\label{g0_L}
    |_0 g (t,x, y, u)|\leq \ell_0  |u|,\,\,  (t,x,y)\in \overline{Q}_T.
    \end{equation}
If necessary, we write the function $g$ in the form (\ref{g0}).

As for the function $H(t, x)$, we will assume that \( D_t^\alpha H(t, x), \, \, \Delta H(t,x) \in C([0, T] \times \overline{G}) \).

Let us apply the overdetermination condition \eqref{overdetermination} to derive a formal representation of the unknown function $h$. To achieve this, we multiply equation \eqref{1} by the function $\omega$ and integrate over the interval $(0, \pi)$. Then
\begin{equation*}
    D_t^{\alpha} H - \Delta H - ( u_{yy}, \omega) = (f(t, x, \cdot), \omega) \,h(t, x) + (g(t, x, \cdot, u), \omega).
\end{equation*}
Taking into account  property (\ref{omega}) of $\omega$, we get $(u_{yy}, \omega) = - (u_y,  \omega')$. Then the desired representation for \( h(t, x) \) can be written as:
\begin{equation}\label{h}
h(t, x) = \frac{D_t^{\alpha} H(t, x) -\Delta H(t, x)-(g(t, x, \cdot, u), \omega) + (u_y, \omega')}{(f(t, x, \cdot), \omega)}.
\end{equation}

Sometimes, when it does not cause confusion, we will write 
\begin{equation}\label{g(u)}
  g(u) := g(t, x, y, u(t, x, y)),\,\, \text{and}\,\,   h(u):=  h(t, x). 
\end{equation}

We now define a weak formulation of the inverse problem \eqref{1}–\eqref{overdetermination}.

\begin{defn}\label{opr1} 
Find a pair of functions \( \{ u(t, x, y), h(t, x) \} \), where \( h(t, x) \) has the form \eqref{h}, and the function \( u(t, x, y) \) satisfies the following conditions:
\begin{enumerate}
    \item \( u \in L_\infty((0, T); L_2(D)) \), \,\, \( u \in L_2(0, T; \dot{W}_2^1(D)) \);
    \item \( D_t^\alpha u \in L_2(0, T; L_2(D)) \);
    \item \( u(0, x, y) = \varphi(x, y) \) a.e.\ in \(D \);
    \item For any \( v \in \dot{W}_2^1(D) \) and almost every \( t \in (0, T] \), the following equality holds:
\end{enumerate}
\begin{equation}\label{equation_main}
    \int_D D_t^\alpha u v \, dx \, dy + \int_D (\nabla u \nabla v + u_y v_y) \, dx \, dy = \int_D f h v \, dx \, dy + \int_D g v \, dx \, dy.
\end{equation}
\end{defn}

We now state the main result of the paper.

\begin{thm}\label{theorem 1}

Let us assume that all the conditions given above regarding the functions $\omega$, $f$, $g$, and $H(t, x)$ are met. Then, the inverse problem has a unique weak solution. Moreover, the following estimates hold:
\begin{equation}\label{uestimate1}
 \sup_{t\in [0, T]} \|  u(t, \cdot, \cdot) \|_{L_2(D)}^2\leq K_1 ,
\end{equation}
 \begin{equation}\label{u_xestimate}
\int_0^t \left(\|  \nabla u (s, \cdot, \cdot)\|_{L_2(D)}^2 +  \, \| u_y (s, \cdot, \cdot) \|_{L_2(D)}^2 \right) ds\leq K_2,     
\end{equation}
\begin{equation}\label{Dalphaestimate}
\int_0^t\| D_s^{\alpha} u^n (s, \cdot, \cdot)\|^2_{L_2(D)}d s
    \leq K_3,
\end{equation}
\begin{equation}\label{h_estimate}
   \sup_{t\in [0, T]} ||h(t, \cdot)||_{L_2(G)}^2
    \leq K_4.
\end{equation}
Here $K_j$, $j=1,2,3, 4,$ are some positive constants, depending on the data of the problem.
\end{thm}

\section{A priori estimates}

To prove the existence of weak solutions, we use the Galerkin method, i.e., first we construct Galerkin's approximations, establish a priori estimates, and pass to the limit.

Let $\psi_k\in \dot{W}^1_2(0, \pi)$ be an orthonormal family in $L_2(0, \pi)$ which linear combinations are dense in $\dot{W}^1_2(0, \pi)$ \cite{Lions}. Given $n\in \mathbb{N}$, let us
consider the $n$-dimensional space $V^n$ spanned by $\psi_1, \cdots, \psi_n$. For each
$n\in \mathbb{N}$, we search for approximate solutions
 to the inverse problem \eqref{1}–\eqref{overdetermination}:
\begin{equation}\label{6}
    u^n(t, x, y) = \sum_{j=1}^n c_j^n(t, x) \psi_j(y),
\end{equation}
where the coefficients $c_1^n(t,x), \cdots , c_n^n(t,x)$ are unknown. 
Due to the conditions of the problem (see Definition \ref{opr1}) the functions $c_j^n(t, x)$ must satisfy the boundary condition
\begin{equation}\label{boundary_c}
c_j^n(t, x) = 0,\,\, x\in \partial G,\,\, t\in (0, T),    
\end{equation}
the equality is understood in the sense of traces.

To find functions $c_j^n(t, x)$ in equation (\ref{equation_main}) instead of $u$ we take $u^n$ and instead of $v$ we take the function $w(x)\psi_k(y)$, $w\in \dot{W}^1_2(G)$. Then
\begin{equation}\label{equation2}
\int_D D_t^\alpha u^n w\, \psi_k\, dx\, dy + \int_D \nabla u^n \nabla w \, \psi_k\, dx\, dy + \int_D u^n_y  \, \psi_k'\,w\, dx\, dy 
\end{equation}
\[
= \int_D (g + f h^n)\, w \, \psi_k\, dx\, dy,\, k=1, 2, \cdots n,
\]
where 
\begin{equation}\label{h_n}
   h(t, x, u^n) := h^n(t, x)
    \end{equation}
    \[
    = \frac{D_t^{\alpha} H(t, x) -\Delta H(t, x)-(g(t, x, \cdot, u^n(t,x, \cdot)), \omega) + (u^n_y(t,x, \cdot)\, , \, \omega')}{(f(t, x, \cdot), \omega)}.
\]
Due to the orthonormality of the system $\{\psi_k\}$, the equation (\ref{equation2}) can be rewritten as
\begin{equation}\label{equation3}
\int_G  D_t^\alpha c_k^n w \, dx + \int_G \nabla c_k^n \, \nabla w \,  dx + \int_G \sum_{j=1}^n a_{j,k}\, c^n_j \, w\, dx  
\end{equation}
\[
= \int_G  (g_k + f_k h^n) w \, dx, \, k=1, 2, \cdots n,
\]
where
\[
a_{j,k} =   (\psi'_j,\, \psi'_k), \, \, f_k(t, x) =  (f (t, x, \cdot), \, \psi_k),
\]
and
\[
g_k(t, x, u^n) = \int_0^\pi  g(t, x, y, \sum_{j=1}^n c_j^n(t,x)\, \psi_j (y) )\, \psi_k(y)\, dy.
\]
The system (\ref{equation3}) of subdiffusion equations is supplemented with the following
Cauchy data
\[
u^n(0, x, y) = \varphi^n (x, y) = \sum_{j=1}^n b_j^n(x) \, \psi_j(y), \,\,  (x, y)\in D,
\]
or
\begin{equation}\label{n_initial}
  c_j^n(0, x) \, = \, b^n_j(x),   \,\, x\in G,
\end{equation}
and assume that 
\begin{equation}\label{n_initial_limit}
\varphi^n (x, y) \to \varphi(x, y)\, \, \text{as} \,\, n\to \infty \,\, \text{in}\,\, \dot{W}_2^1(D).
    \end{equation}

    Equation (\ref{equation3}), where $h^n$ has the form (\ref{h_n}), with the boundary and initial conditions (\ref{boundary_c}) and (\ref{n_initial}) is a weak statement of the following initial-boundary value problem for the system of subdiffusion equations with respect to functions $c_k^n (t, x)$, $ k=1, 2, \cdots,  n,$
    \begin{equation}\label{Cauchy}
      D_t^\alpha c_k^n (t, x)  - \Delta c_k^n (t, x) + \sum_{k=1}^n a_{j,k} c_k^n (t, x) = g_k (t, x, u^n) + f_k (t, x) h^n (t, x).
    \end{equation} 
    Now we will show that the right-hand sides of these equations satisfy the Lipschitz condition. 
    
    Let $u^n = \sum_1^n c_j^n \psi_j$ and $v^n = \sum_1^n d_j^n \psi_j$. Apply the Lipschitz condition, the Cauchy-Schwarz inequality, and the Parseval equality to obtain.
    \[
    |g_k(t, x, u^n) - g_k(t, x, v^n)| \leq \int_0^\pi |g(t, x, y, u^n) - g(t, x, y, v^n)| |\psi_k(y)| dy
    \]
   \[
  \leq \ell_0 \int_0^\pi |u^n -v^n||\psi_k(y)| dy \leq \ell_0 ||u^n -v^n|| = \ell_0 \left( \sum_{j=1}^n |c_j^n - d_j^n|^2\right)^{\frac{1}{2}}\leq \ell_0 \sum_{j=1}^n |c_j^n - d_j^n|.
   \]
   Therefore, $g_k(t, x, u^n)$ satisfies the Lipschitz condition. Similarly, we get
   \[
  |( (g(t, x, \cdot, u^n) - (g(t, x, \cdot, v^n)), \omega)|\leq \ell_0 ||\omega|| \left( \sum_{j=1}^n |c_j^n - d_j^n|^2\right)^{\frac{1}{2}}\leq \ell_0 ||\omega||\sum_{j=1}^n |c_j^n - d_j^n|.
   \]
   Next we obtain
   \[
   |( u_y^n - v_y^n , \omega)|\leq \sum_{j=1}^n |c_j^n - d_j^n| \int_0^\pi |\psi_j'(y) \omega(y)| dy \leq ||\omega||\sum_{j=1}^n |c_j^n - d_j^n|.
   \]
   It is not difficult to verify that from the last two estimates it follows that $f_k (t, x) h^n (t, x):=  f_k (t, x) h(t, x, u^n)$ satisfies the Lipschitz condition with respect to $c_k^n$, which is what was required to be proved.

   In the case where the right-hand side of equation (\ref{Cauchy}) satisfies the Lipschitz condition, the correctness of the initial-boundary value problem (\ref{Cauchy}), (\ref{boundary_c}) and (\ref{n_initial}) is well known (see, for example, \cite{Gal}, Section 3).

Having made sure that $c_k^n(t, x)$ exist for all $k$, we proceed to estimating the function $u^n$ and its derivatives involved in problem (\ref{equation2}) - (\ref{h_n}).
\begin{lem}\label{un_estimateLemam} There is a positive constant $K_1$ such that for all $n$
\begin{equation}\label{un_estimate}
    \| u^n(t, \cdot, \cdot) \|_{L_2(D)}^2\leq K_1 .
\end{equation}
\end{lem}
\begin{proof}
In equation (\ref{equation2}) instead of $w(x)$ we take $c_k^n(t, x)$ and sum up by $k$ from $1$ to $n$. Then
\[
\int_D D_t^\alpha u^n u^n \, dx\, dy + \| \nabla u^n \|_{L_2(D)}^2 +  \| u^n_y \|_{L_2(D)}^2
= \int_D (g + f h^n)\, u^n\, dx\, dy.
\]
For the integral on the left-hand side, we apply Alikhanov’s estimate (see Lemma \ref{Alixan1}), obtaining:
\begin{equation}\label{Du^n}
\frac{1}{2} D_t^\alpha \| u^n \|_{L_2(D)}^2 + \| \nabla u^n \|_{L_2(D)}^2 +  \| (u^n)_y \|_{L_2(D)}^2
= \int_D (g + f h^n)\, u^n\, dx\, dy.
\end{equation} 
Let us evaluate the right side separately. We have (see (\ref{g0}) and (\ref{g0_L}))
\[
\left|\int_D g(t,x,y, u^n)\, u^n\, dx\, dy\right| \leq \int_D |g(t,x,y, 0)\, u^n|\, dx\, dy + \int_D |_0\, g(t,x,y, u^n)\, u^n|\, dx\, dy
\]
\begin{equation}\label{gu^n}
    \leq \frac{1}{2} ||g(t, \cdot, \cdot, 0)||_{L_2(D)}^2 + \frac{1}{2} ||u^n||_{L_2(D)}^2 + \ell_0 \,  ||u^n||_{L_2(D)}^2.
\end{equation}

On the other hand, using the estimate $2\, ab \leq \varepsilon a^2 + \frac{1}{\varepsilon} b^2$, we will have  
\begin{equation}\label{fu^n}
    \left|\int_D f h^n\, u^n\, dx\, dy\right |\leq \frac{\varepsilon}{2} ||f h^n||_{L_2(D)}^2 + \frac{1}{2\varepsilon} ||u^n||_{L_2(D)}^2,
\end{equation}
the number $\varepsilon>0$ will be chosen later.
Next, we estimate the first term on the right-hand side. We have (see (\ref{fomega}) and (\ref{h}))
\[
\left|\left|\frac{D_t^{\alpha} H(t, x) - \Delta\, H(t, x)}{(f(t, x, \cdot), \omega)} f(t,x, y) \right|\right|_{L_2(D)}^2\leq f_M^2 \left(||D_t^{\alpha} H(t, \cdot)|_{L_2(G)}^2 + ||\Delta\, H(t, \cdot)||_{L_2(G)}^2\right).
\]
We also have
\begin{equation}\label{gestimate}
    |g(t,x,y, u^n),\, \omega)|^2\leq 2 | g(t,x,y, 0),\, \omega)|^2 + 2\, | _0\, g(t,x,y, u^n),\, \omega)|^2 
\end{equation}
\[
\leq 2 ||g(t,x, \cdot, 0)||^2 ||\omega||^2 +2 \ell^2_0 ||u^n (t,x, \cdot)||^2 ||\omega||^2. 
\]
Therefore
\[
\left|\left|\frac{g(t,x,\cdot, u^n),\, \omega)}{(f(t, x, \cdot), \omega)} f(t,x, y) \right|\right|_{L_2(D)}^2\leq 2\, f_M^2 \, ||\omega||^2\left(||g(t,\cdot, \cdot, 0)||_{L_2(D)}^2 + \ell^2_0 ||u^n ||_{L_2(D)}^2\right).
\]
If we use for $(u_y^n, \, \omega')$ Cauchy–Schwarz inequality (see (\ref{fomega})), then
\[
\left|\left|\frac{(u_y^n, \, \omega')}{(f(t, x, \cdot), \omega)} f(t,x, y) \right|\right|_{L_2(D)}^2\leq  f_M^2 \, ||\omega'||^2\, ||u^n_y ||_{L_2(D)}^2.
\]
From the last four estimates, in particular, we obtain
\begin{equation}\label{fh^n_integral}
  ||f h^n||_{L_2(D)}^2\leq f_M^2 \left(H_0^2(t)+ 2 ||\omega||^2 (  ||g(t,\cdot, \cdot, 0)||_{L_2(D)}^2 + \ell^2_0 ||u^n ||_{L_2(D)}^2) + ||\omega'||^2 ||u_y^n|||_{L_2(D)}^2\right),
\end{equation}
where
\[
H_0^2(t) = ||D_t^{\alpha} H(t, \cdot)|_{L_2(G)}^2 + ||\Delta\, H(t, \cdot)||_{L_2(G)}^2.
\]

Due to inequality $(a+b+c)^2 \leq 3( a^2 + b^2 + c^2)$, we can rewrite (\ref{fu^n}) as
\[
  \left|\int_D f h^n\, u^n\, dx\, dy\right |\leq \frac{3}{2} \varepsilon f_M^2 \left[ H_0^2(t) +2 ||\omega||^2 ||g(t, \cdot, \cdot, 0)||^2_{L_2(D)}\right]
\]
\[
+\left(3\, \varepsilon\, \ell_0^2\, f_M^2\, ||\omega||^2 +\frac{1}{2 \varepsilon}\right) ||u^n||_{L_2(D)}^2 + \frac{3}{2} \varepsilon f_M^2 ||\omega'||^2 ||u_y^n||_{L_2(D)}^2.
\]
Now we choose the number $\varepsilon$ so that the coefficient before \(||u_y^n||_{L_2(D)}^2\) is $\frac{1}{2}$, i.e., $\varepsilon = (3 f_M^2 ||\omega'||^2)^{-1}$. Then we get the final estimate
\begin{equation}\label{fu^nNew}
    \left|\int_D f h^n\, u^n\, dx\, dy\right |\leq \frac{1}{2 \, ||\omega'||^2} \left[ H_0^2(t) +2 ||\omega||^2 ||g(t, \cdot, \cdot, 0)||^2_{L_2(D)}\right]
    \end{equation}
    \[
    + \left[\frac{\ell_0^2 ||\omega||^2}{||\omega'||^2} + \frac{3 f_M^2 ||\omega'||^2}{2}\right] ||u^n||_{L_2(D)}^2 + \frac{1}{2} ||u_y^n||_{L_2(D)}^2.
    \]

Taking into account the estimate (\ref{gu^n}), we transform equality (\ref{Du^n}) into the following inequality
\begin{equation}\label{Du^n_1}
\frac{1}{2} D_t^\alpha \| u^n \|_{L_2(D)}^2 + \| \nabla u^n \|_{L_2(D)}^2 +  \frac{1}{2} \| u^n_y \|_{L_2(D)}^2
\end{equation}
\[
\leq \left(\frac{1}{2} +  \frac{||\omega||^2}{||\omega'||^2}\right) ||g(t, \cdot, \cdot, 0)||_{L_2(D)}^2 + \frac{H_0^2(t)}{2 ||\omega'||^2} + L_\omega^2 \| u^n \|_{L_2(D)}^2,
\]
where
\[
L_\omega^2 = \frac{\ell_0^2 ||\omega||^2}{||\omega'||^2} + \frac{3 f_M^2 ||\omega'||^2}{2} + \ell_0 + \frac{1}{2}.
\]
If we discard the last two terms on the left-hand side of (\ref{Du^n_1}), then
\[
D_t^\alpha \| u^n \|_{L_2(D)}^2 \leq 2\,L_\omega^2\,\| u^n \|_{L_2(D)}^2 +c_2(t),
\]
where
\[
c_2(t) = \left(1 +  \frac{2||\omega||^2}{||\omega'||^2}\right) ||g(t, \cdot, \cdot, 0)||_{L_2(D)}^2 + \frac{H_0^2(t)}{ ||\omega'||^2}.
\]
Apply Lemma \ref{Alixan2L} to obtain
\[
\| u^n \|_{L_2(D)}^2\leq \| \varphi^n \|_{L_2(D)}^2 \,E_\alpha(2\,L_\omega^2\, t^\alpha) + \Gamma(\alpha) \,E_{\alpha, \alpha}(2\,L_\omega^2\, t^\alpha) \, J_t^\alpha c_2(t).
\]
Due to the boundedness  of the Mittag-Leffler function (see (\ref{E})) and the estimate (\ref{integral}), we obtain the final estimate
\begin{equation}\label{un_final}
    \| u^n(t, \cdot, \cdot) \|_{L_2(G)}^2\leq M\| \varphi^n \|_{L_2(D)}^2  + \frac{MT^\alpha}{\alpha} \,\max_{[0,T]}c_{2}(t) .
\end{equation}
\end{proof}
The lemma is proved.

Let us obtain estimates for the derivatives of $u^n$. 
\begin{lem}\label{nablaestimate} There is a positive constant $K_2$ such that for all $n$
\begin{equation}\label{un_nablaestimate}
 \int_0^t \left(\|  \nabla u^n (s, \cdot, \cdot)\|_{L_2(D)}^2 +  \, \| u^n_y (s, \cdot, \cdot) \|_{L_2(D)}^2 \right) ds\leq K_2.     
 \end{equation}
\end{lem}
\begin{proof}

To prove the lemma, we will apply the operator $J_t^\alpha$ to both parts of inequality (\ref{Du^n_1}) and use Lemma \ref{J}. Then
\begin{equation}\label{Du^n_final}
\| u^n \|_{L_2(G)}^2 + \, 2\, J_t^\alpha \| \nabla u^n \|_{L_2(D)}^2 +  \, J_t^\alpha \| u^n_y \|_{L_2(D)}^2 \leq 2\,L_\omega^2\,J_t^\alpha\, \| u^n \|_{L_2(D)}^2 +J_t^\alpha\, c_2(t).
\end{equation}
First, we omit the first term on the left-hand side of this inequality, then, applying the estimates (\ref{integral}) and (\ref{un_final}), we obtain
\[
\frac{T^{\alpha -1 }}{\Gamma(\alpha)\, }\int_0^t \left(\| 2 \, \nabla u^n (s, \cdot, \cdot)\|_{L_2(D)}^2 +  \, \| u^n_y (s, \cdot, \cdot) \|_{L_2(D)}^2 \right) ds 
\]
\[
\leq 
 \frac{T^\alpha}{\alpha\, \Gamma (\alpha)}2\,L_\omega^2\,\left[M\| \varphi^n \|_{L_2(D)}^2  + \frac{MT^\alpha}{\alpha} \,\max_{[0,T]}c_{2}(t) \right] + \frac{T^\alpha}{\alpha\, \Gamma (\alpha)} \max_{[0,T]}\, c_2(t),
 \]
 or finally
\begin{equation}\label{u_y}
 \int_0^t \left(\| 2 \, \nabla u^n (s, \cdot, \cdot)\|_{L_2(D)}^2 +  \, \| u^n_y (s, \cdot, \cdot) \|_{L_2(D)}^2 \right) ds     
 \end{equation}
\[
\leq 
 \frac{T}{\alpha}\left[2\,L_\omega^2\,M\| \varphi^n \|_{L_2(D)}^2  + \left(\frac{2\,L_\omega^2MT^\alpha}{\alpha} +1\right)\,\max_{[0,T]}c_{2}(t) \right]. 
 \]
 This is the statement of the lemma.
\end{proof}
Estimate (\ref{Du^n_final}), in particular, implies the following estimate necessary for applying Lemma \ref{li} on compactness.
\begin{equation}\label{estimateLi}
    \sup\limits_{t\in (0, T)}  J_t^\alpha(\| \nabla u^n \|_{L_2(D)}^2 +  \, \| u^n_y \|_{L_2(D)}^2) \leq C_1,\,\, C_1 >0.
\end{equation}

Let us move on to estimating the fractional derivative $D_t^\alpha u^n$.
\begin{lem}\label{DunMain} There is a positive constant $K_3$ such that for all $n$
\begin{equation}\label{Destimate}
  \int_0^t\| D_\tau^{\alpha} u^n (\tau, \cdot, \cdot)\|^2_{L_2(D)}d \tau
    \leq K_3.  
\end{equation}
\end{lem}
\begin{proof}
In equation (\ref{equation2}) instead of $w(x)$ we take $D_t^\alpha c_k^n(t, x)$ and sum up by $k$ from $1$ to $n$. Then
\[
\| D_t^\alpha u^n \|_{L_2(D)}^2 + \int_D \nabla u^n D_t^\alpha \nabla u^n \, dx\, dy +  \int_D u_y^n D_t^\alpha u_y^n \, dx\, dy = \int_D (g + f h^n)\, D_t^\alpha\, u^n\, dx\, dy.
\]
Let us apply A. Alikhanov's estimate (see Lemma \ref{Alixan1}) in the second and third integrals on the left-hand side of the equality. Then
\begin{equation}\label{D_t}
\| D_t^{\alpha} u^n \|^2_{L_2(D)} + \frac{1}{2} D_t^{\alpha} \| \nabla u^n \|^2_{L_2(D)} +  \frac{1}{2} D_t^{\alpha} \| u_y^n \|^2_{L_2(D)} \leq \left| \int_D (g + f h^n)\, D_t^\alpha\, u^n\, dx\, dy\right |. 
\end{equation}
Now we apply the operator $J_t^\alpha$ to both parts of inequality (\ref{D_t}) and use Lemma \ref{J}. Then
\[
J_t^\alpha\| D_t^{\alpha} u^n \|^2_{L_2(D)}+ \frac{1}{2} \| \nabla u^n \|^2_{L_2(D)} +  \frac{1}{2} \| u_y^n \|^2_{L_2(D)}
\]
\[
\leq \frac{1}{2} \| \nabla \varphi^n \|^2_{L_2(D)} + \frac{1}{2} \| \varphi_y^n \|^2_{L_2(D)}+J_t^\alpha\, \left| \int_D (g + f h^n)\, D_t^\alpha\, u^n\, dx\, dy\right |.
\]
First, we omit the second and third terms on the left-hand side of this inequality. Then, 
\begin{equation}\label{JD}
J_t^\alpha\| D_t^{\alpha} u^n \|^2_{L_2(D)}\leq \frac{1}{2} \| \nabla \varphi^n \|^2_{L_2(D)} + \frac{1}{2} \| \varphi_y^n \|^2_{L_2(D)}+J_t^\alpha\, \left| \int_D (g + f h^n)\, D_t^\alpha\, u^n\, dx\, dy\right |.
\end{equation}
For the last integral on the right-hand side, we have (see the properties  (\ref{g0}) and (\ref{g0_L}))
\[
\left|\int_D g(t,x,y, u^n)\, D_t^\alpha\, u^n\, dx\, dy\right|  \leq \left|\int_D g(t,x,y, 0)\, D_t^\alpha\, u^n\, dx\, dy\right|+ \ell_0\int_D |u^n\, D_t^\alpha\, u^n|\, dx\, dy.
\]
Let us estimate the integrals on the right-hand side using the Cauchy–Schwarz inequality $2\, ab \leq \varepsilon a^2 + \frac{1}{\varepsilon} b^2$, where we take $\varepsilon=\frac{1}{2}$ in the first integral and $\varepsilon=\frac{1}{2\, \ell_0}$ in the second integral. As a result, we have
\[
\left|\int_D g(t,x,y, u^n)\, D_t^\alpha\, u^n\, dx\, dy\right|  \leq  ||g(t, \cdot, \cdot, 0)||_{L_2(D)}^2 + \frac{1}{2} ||D_t^\alpha\,u^n||_{L_2(D)}^2 + \ell^2_0 \,  ||u^n||_{L_2(D)}^2.
\]
On the other hand, in the same way, we get
\begin{equation}\label{fDu^n}
    \left|\int_D f h^n\, D_t^\alpha\,u^n\, dx\, dy\right |\leq  ||f h^n||_{L_2(D)}^2 + \frac{1}{4} ||D_t^\alpha\,u^n||_{L_2(D)}^2.
\end{equation}
For $||f h^n||_{L_2(D)}^2$ we have (see (\ref{fh^n_integral}))
\[
 ||f h^n||_{L_2(D)}^2\leq f_M^2 \left(H_0^2(t)+ 2 ||\omega||^2 (  ||g(t,\cdot, \cdot, 0)||_{L_2(D)}^2 + \ell^2_0 ||u^n ||_{L_2(D)}^2) + ||\omega'||^2 ||u_y^n|||_{L_2(D)}^2\right).
 \]
Now returning to the estimate (\ref{JD}), we will have
\[\frac{1}{4}J_t^\alpha\| D_t^{\alpha} u^n \|^2_{L_2(D)}\leq \frac{1}{2} \| \nabla \varphi^n \|^2_{L_2(D)} + \frac{1}{2} \| \varphi_y^n \|^2_{L_2(D)}+J_t^\alpha\,  \left(||g(t, \cdot, \cdot, 0)||_{L_2(D)}^2 +  \ell^2_0 \,  ||u^n||_{L_2(D)}^2 \right )
\]
\[
+ f_M^2\, J_t^\alpha  \left(H_0^2(t)+ 2 ||\omega||^2 (  ||g(t,\cdot, \cdot, 0)||_{L_2(D)}^2 + \ell^2_0 ||u^n ||_{L_2(D)}^2) + ||\omega'||^2 ||u_y^n|||_{L_2(D)}^2\right).
\]
Apply the estimate   (\ref{integral}), to obtain 
\[
\int_0^t\| D_\tau^{\alpha} u^n (\tau, \cdot, \cdot)\|^2_{L_2(D)}d \tau\leq  T^{1-\alpha}\, \Gamma(\alpha) \, J_t^\alpha\| D_t^{\alpha} u^n \|^2_{L_2(D)}.
\]
Now, from estimates  of $u^n$, $u_y^n$ in (\ref{un_final}) and (\ref{u_y}) the assertion of the lemma follows.
\end{proof}

In conclusion of this section, we evaluate $h^n$. 
\begin{lem}\label{hestimatelemma} There is a positive constant $K_4$ such that for all $n$
\begin{equation}\label{hestimate}
 \sup_{t\in(0, T)} ||h^n(t, x)||_{L_2(G)}^2
    \leq K_4.  
\end{equation}
\end{lem}
\begin{proof}
By the definition of $h^n$ (see (\ref{h_n})), we have
\[
||h^n(t, x)||_{L_2(G)}^2
    \leq  3\left|\left|\frac{D_t^{\alpha} H(t, x) -\Delta H(t, x)}{(f(t, x, \cdot), \omega)}\right|\right|_{L_2(G)}^2
\]
\begin{equation}\label{hestimate2}
    + 3\left|\left|\frac{(g(t, x, \cdot, u^n(t,x, \cdot)), \omega)}{(f(t, x, \cdot), \omega)}\right|\right|_{L_2(G)}^2 + 3\left|\left|\frac{(u^n_y(t,x, \cdot)\, , \, \omega')}{(f(t, x, \cdot), \omega)}\right|\right|_{L_2(G)}^2.
\end{equation}
The second term on the right is estimated using the relation (\ref{gestimate}) and to estimate the third term, we apply the Cauchy-Schwartz inequality. As a result, we obtain
\[
||h^n(t, x)||_{L_2(G)}^2
    \leq A_1 + A_2 ||u^n||_{L_2(D)}^2 + A_3 ||u_y^n||_{L_2(D)}^2,
    \]
where $A_j$ does not on $n$. Now, from estimates (\ref{un_final}) and (\ref{u_y}), the assertion of the lemma follows. \end{proof}

\section{Convergence and Proof of Theorem \ref{theorem 1}}

From the estimates (\ref{un_estimate}),(\ref{un_nablaestimate}), and (\ref{Destimate})  it follows that the sequences $u^n(\xi, x, y)$, $u_y^n(\xi, x, y)$, $\nabla u^n(\xi, x, y)$ and $D^\alpha_\xi u^n(\xi, x, y)$ are bounded in the norm of the space $L_2(Q_t)$, for any $t\in (0, T]$. Therefore, from the sequence $u^n(\xi, x, y)$ one can choose a subsequence $u^\mu(\xi, x, y)$ such that the following weak convergences in $L_2(Q_t)$ hold (note that in our case, for example, the weak limit of the Caputo derivatives is the Caputo derivative of the weak limit):
\begin{equation}\label{converUx}
\begin{cases}
u^\mu (\xi, x, y)\rightharpoonup  u (\xi, x, y),\\
u_y^\mu (\xi, x, y)\rightharpoonup  u_y (\xi, x, y),\\
\nabla\,u^\mu (\xi, x, y)\rightharpoonup  \nabla\, u (\xi, x, y), \\
D_\xi^\alpha u^\mu (\xi, x, y) \rightharpoonup D_\xi^\alpha u (\xi, x, y).
\end{cases}
\end{equation}

 Now we apply Lemma \ref{li} with $M= \dot{W}_2^1 (D)$ and $B=Y= L_2(D)$ to obtain strong convergence. To do this, we note that the embedding $\dot{W}_2^1 (D) \hookrightarrow \hookrightarrow  L_2(D)$  is compact and the Riemann-Liouville integrals satisfy the estimate (\ref{estimateLi}) (see also (\ref{Destimate})). By this lemma, we have
 \begin{equation}\label{strong}
u^\mu \to u \,\,\text{strongly\,\, in}  \,\, L_2(Q_T), \,\, \text{as}\,\, \mu \to \infty,   
 \end{equation}
 and
 \begin{equation}\label{almost}
u^\mu \to u \,\,\text{a.e.\,\, in}  \,\, Q_T, \,\, \text{as}\,\, \mu \to \infty,   
 \end{equation}

Let us consider the components of function $h^\mu$ separately. If we take as $w(\tau, x, y) \in L_2(Q_t)$ a function $W(\tau, x) V(y)$ with $W\in L_2(0, t; L_2(G))$ and $V\in L_2(0, \pi)$, then by virtue of the estimates (\ref{fomega}) we have $f_\omega \, \omega' \, W \in L_2 (Q_t)$. Therefore, passing to the limit $\mu \to \infty$ in the next integral
\[
\int_0^\pi V(y) \int_0^t \int_G \int_0^\pi  u^\mu_\eta(\tau,x, \eta)\, f_\omega (\tau, x, y)\, \omega'(\eta) W(\tau, x) \, d\tau\, dx\, d\eta,
\]
we will have (see the second line in (\ref{converUx}))
\[
 f_\omega (t, x, y)\,(u_y^\mu, \omega')   \rightharpoonup   f_\omega (t, x, y)\,(u_y, \omega'), \,\, \text{weak\,\, in} \,\, L_2(Q_t)\,\, \text{as}\,\, \mu\to \infty.
\]
On the other hand, we have (see (\ref{g_Lipschitz}) and (\ref{g(u)}))
\[
    ||f_\omega (g(u^\mu) - g(u), \omega)||^2_{L_2(Q_t)} \leq f_M^2 \ell_0^2 ||\omega||^2 ||u^\mu- u||^2_{L_2(Q_t)} \to 0, \,\, \text{as}\,\, \mu\to \infty.
\]
From the last two relations, it follows (see  (\ref{g(u)}))
\begin{equation}\label{h_mu}
f_\omega (t, x, y)\,h(u^\mu)   \rightharpoonup   f_\omega (t, x, y)\,h(u), \,\, \text{weak\,\, in} \,\, L_2(Q_t)\,\, \text{as}\,\, \mu\to \infty.
\end{equation}

Again, due to the Lipschitz condition, with respect to the function $g(u^\mu)$ we have
\begin{equation}\label{g_mu}
    ||g(u^\mu) - g(u)||^2_{L_2(Q_t)} \leq  \ell_0^2  ||u^\mu- u||^2_{L_2(Q_t)} \to 0, \,\, \text{as}\,\, \mu\to \infty.
\end{equation}

Now let us multiply equation (\ref{equation_main}) by $w\in L_2(0, T)$, then integrate over $(0,t)$ and rewrite the resulting equation with respect to $u^\mu$. Then
\[
    \int_0^t\int_D D_\tau^\alpha u^\mu v w\, dx \, dy d\tau+ \int_0^t\int_D (\nabla u^\mu \nabla v  + u^\mu_y v_y) w\, dx \, dy d\tau 
    \]
\[
        = \int_0^t\int_D f h(u^\mu) v w \, dx \, dy  d\tau + \int_0^t\int_D g(u^\mu) v w\, dx \, dy d\tau ,
   \]
where $v\in \dot{W}_2^1(D)$. 
Taking into account (\ref{h_mu}) and (\ref{g_mu}) we pass here to the limit  $\mu \to \infty$. Then we will have
\begin{equation}\label{equation_mainT}
    \int_0^t \int_D D_t^\alpha u v w\, dx \, dy d\tau +  \int_0^t\int_D (\nabla u \nabla v + u_y v_y) \, w\, dx \, dy d\tau
    \end{equation}
    \[
    =  \int_0^t\int_D f h (u) v w\, dx \, dy d\tau +  \int_0^t\int_D g(u) v w\, dx \, dy d\tau.
\]

Let us note the obvious fact that if $s(t)$ is integrable in any subset $(0,t)$ of $[0,T]$ and $\int_0^t s(\tau) d\tau =0$, then $s(t)=0$ is almost everywhere on $[0,T]$. Since $w\in L_2(0, T)$ is an arbitrary function, then the last equality coincides with (\ref{equation_main}). In other words, the pair of functions $\{u(t,x,y), h(t,x)\}$ defined as a limit of (\ref{6}) and (\ref{h}) is a weak solution of the inverse problem, i.e., it satisfies all conditions of Definition \ref{opr1}. 

Finally, by the lower semicontinuity of the norms, we can combine the
estimates (\ref{un_estimate}),(\ref{un_nablaestimate}), (\ref{Destimate}),  and (\ref{hestimate}) with the convergence results stated at
(\ref{converUx}) to prove (\ref{uestimate1}) - (\ref{h_estimate}).

To prove Theorem \ref{theorem 1}, it remains to show the uniqueness of the solution of the inverse problem. 

Suppose that there are two weak solutions to the inverse problem: \( (u_1, h(u_1)) \) and \( (u_2, h(u_2)) \) (see (\ref{g(u)})).  
Let us first write equation (\ref{equation_main}) relative to $u_1$ and $u_2$ and take the difference. Then let us take $u: = u_1 - u_2$ as a test function and rewrite the resulting equation:
\begin{equation}\label{equation_main_uniq}
    \int_D D_t^\alpha u u \, dx \, dy + ||\nabla u||^2_{L_2(D)} + || u_y||^2_{L_2(D)} \, dx \, dy
    \end{equation}
    \[
    = \int_D f (h(u_1) -h(u_2)) u \, dx \, dy + \int_D ( g(u_1) - g(u_2)) u \, dx \, dy.
    \]
    For the second integral on the right-hand side, we have
\[
\left|\int_D ( g(u_1) - g(u_2)) u \, dx \, dy\right|\leq \ell_0 ||u||^2_{L_2(D)}.
\]
With respect to the first integral on the right-hand side, we write (see (\ref{fu^n}))
\[
\left|\int_D  f (h(u_1) -h(u_2)) u\, dx\, dy\right |\leq \frac{\varepsilon}{2} || f (h(u_1) -h(u_2))||_{L_2(D)}^2 + \frac{1}{2\varepsilon} ||u||_{L_2(D)}^2,
\]
where $\varepsilon = (3 f_M^2 ||\omega'||^2)^{-1}$. Repeating the same reasoning as in the proof of estimate (\ref{fu^nNew}), we will have
 \[\left|\int_D  f (h(u_1) -h(u_2)) u\, dx\, dy\right |\leq    \left[\frac{\ell_0^2 ||\omega||^2}{||\omega'||^2} + \frac{3 f_M^2 ||\omega'||^2}{2}\right] ||u||_{L_2(D)}^2 + \frac{1}{2} ||u_y||_{L_2(D)}^2.
    \]
    Now let us rewrite the equality (\ref{equation_main_uniq}) as
\[
 \int_D D_t^\alpha u u \, dx \, dy + ||\nabla u||^2_{L_2(D)} + \frac{1}{2} || u_y||^2_{L_2(D)} \, dx \, dy
\]
\[
\leq \left[\frac{\ell_0^2 ||\omega||^2}{||\omega'||^2} + \frac{3 f_M^2 ||\omega'||^2}{2} + \ell_0 \right] ||u||_{L_2(D)}^2.
\]
For the integral on the left-hand side, we apply Alikhanov's estimate (see Lemma \ref{Alixan1}), and discard the next two terms. Then
\[
D_t^\alpha \| u(t, \cdot, \cdot) \|_{L_2(D)}^2 \leq 2\,\left[\frac{\ell_0^2 ||\omega||^2}{||\omega'||^2} + \frac{3 f_M^2 ||\omega'||^2}{2} + \ell_0 \right] ||u (t, \cdot, \cdot)||_{L_2(D)}^2, \,\, t\in (0, T],
\]
and \(\| u (0, \cdot, \cdot)\|_{L_2(D)}^2 =0\). By virtue of Lemma \ref{Alixan2L} it follows that $||u (t, \cdot, \cdot)||_{L_2(D)}^2 \equiv 0$ for all $t\in (0,\, T]$. Consequently $u_1 (t, x, y) = u_2 (t, x, y)$ almost everywhere in $Q_T$. By virtue of the definition (\ref{h}) of the function $h$ we obviously have $h(t, x) =0$, for almost all $(t, x) \in (0, T)\times G$. Thus, the uniqueness of the solution of the inverse problem, and hence Theorem \ref{theorem 1} is completely proved.

\begin{rem}\label{zam.3} Note that the Theorem \ref{theorem 1} remains valid for diffusion equations as well. In this case, the usual equality $\frac{d}{dt} u^2= 2 u \frac{d}{dt} u$ should be used instead of Alikhanov's estimate.
\end{rem}

\section{Conclusions}
 
In this paper, we study a new inverse problem for nonlinear subdiffusion equations, namely, the problem of determining a source function that depends on both time and a subset of spatial variables. The nonlinearity of the equation is that the equation involves a function of the solution of the equation: $g(t,x,y, u)$. To our knowledge, such a problem has not been studied before, even for the classical diffusion equation. To solve this problem, we apply the Galerkin method and the method of a priori estimates. To prove the convergence of the corresponding series, we apply an analogue of the Aubin–Lions compactness lemma in the case of fractional derivatives. The existence and uniqueness of a weak solution to the inverse problem are established, and coercive estimates are proven.
The elliptic part of the equation is given by $\Delta_x u + u_{yy}$, $x\in \mathbb{R}^n$, and the unknown function depends on a subset of spatial variables, denoted by $h(t, x)$. The proposed method can be extended without significant changes to the case when $y \in \mathbb{R}^m$ (see \cite{AshurovMuhiddinovanew1}).

The choice of constructing a weak solution is motivated by the relative simplicity of obtaining a priori estimates in this structure. By increasing the smoothness of the solution, one can prove the existence of a strong solution to the inverse problem (see \cite{Hompush}). However, this will be the focus of a future article.

\

\begin{center}
ACKNOWLEDGEMENTS    
\end{center}

The authors are grateful to Sh. A. Alimov, and Z. Sabirov for discussions of
these results. The authors also acknowledge financial support from the Ministry of Innovative Development of the Republic of Uzbekistan, Grant No F-FA-2021-424.


\end{document}